\documentclass[a4paper]{article}
\usepackage{amssymb,amsmath,amsfonts,amsthm,graphicx,verbatim}%setspace,enumerate}
\usepackage{stmaryrd}
\usepackage{hyperref}
\theoremstyle{plain}
\newtheorem{proposition}{Proposition}[section]
\newtheorem{theorem}[proposition]{Theorem}

\def\msig{\mc{F}^\sigma_m}
\def\mc{\mathcal}

\def\msig{\mc{F}^\sigma_m}
\def\mc{\mathcal}

\def\F{\mathcal{F}}
\def\tr{\textrm}

\def\progCE{\textsf{HypercubeEdgeDensityChecker}}
\def\progCV{\textsf{HypercubeVertexDensityChecker}}
\def\progCPE{\textsf{PartialHypercubeEdgeDensityChecker}}

\begin{document}

\title{Tur\'an densities of hypercubes}

%\author{Rahil Baber\thanks{Department of Mathematics, UCL, London, WC1E 6BT, UK. Email: rahilbaber@hotmail.com.}}
\author{Rahil Baber}

\date\today

\maketitle

\begin{abstract}
In this paper we describe a number of extensions to Razborov's
semidefinite flag algebra method. We will begin by showing how to
apply the method to significantly improve the upper bounds of edge
and vertex Tur\'an density type results for hypercubes. We will then
introduce an improvement to the method which can be applied in a
more general setting, notably to $3$-uniform hypergraphs, to get a
new upper bound of $0.5615$ for $\pi(K_4^3)$.

For hypercubes we improve Thomason and Wagner's result on the upper
bound of the edge Tur\'an density of a 4-cycle free subcube to
$0.60318$ and Chung's result on forbidding 6-cycles to $0.36577$. We
also show that the upper bound of the vertex Tur\'an density of
$\mc{Q}_3$ can be improved to $0.76900$, and that the vertex Tur\'an
density of $\mc{Q}_3$ with one vertex removed is precisely $2/3$.

\end{abstract}

\section{Introduction}\label{SUBSEC:IntroHypercube}

Razborov's flag algebra method introduced in \cite{RF} has proven to
be an invaluable tool in finding upper bounds of Tur\'an densities
of hypergraphs. Many results have been found through its
application, for some such results and descriptions of the
semidefinite flag algebra method as applied to hypergraphs see
\cite{B11}, \cite{BT}, \cite{BTNew}, \cite{FV}, \cite{R4}. Later in
this paper we will describe how to improve the bounds Razborov's
method is able to attain, in particular by using partially defined
graphs. However, to begin with we will extend Razborov's method from
hypergraphs to hypercubes.

An $n$-dimensional hypercube $\mc{Q}_n$ is a $2$-graph with $2^n$
vertices. Setting $V(\mc{Q}_n)=\{0,1,\ldots,2^n-1\}$ we can define
$E(\mc{Q}_n)$ as follows: $v_1v_2\in E(\mc{Q}_n)$ if and only if
$v_1$ differs from $v_2$ by precisely one digit in their binary
representations. For example $E(\mc{Q}_2)=\{01, 02, 13, 23\}$ (see
also Figure \ref{Fig:Layers}). It is easy to see that the binary
representations of the vertices indicate the coordinates of the
vertices of a unit hypercube in $\mathbb{R}^n$. Let us also define
the layers of a hypercube which will be useful later. \emph{Layer}
$m$ of $\mc{Q}_n$ consists of all vertices in $V(\mc{Q}_n)$ which
have $m$ digits that are one in their binary representations. For
example in $\mc{Q}_3$, layer $0=\{0\}$, layer $1=\{1,2,4\}$, layer
$2=\{3,5,6\}$, and layer $3=\{7\}$, see Figure \ref{Fig:Layers}.

\begin{figure}[tbp]
\begin{center}
\includegraphics[height=3.6cm]{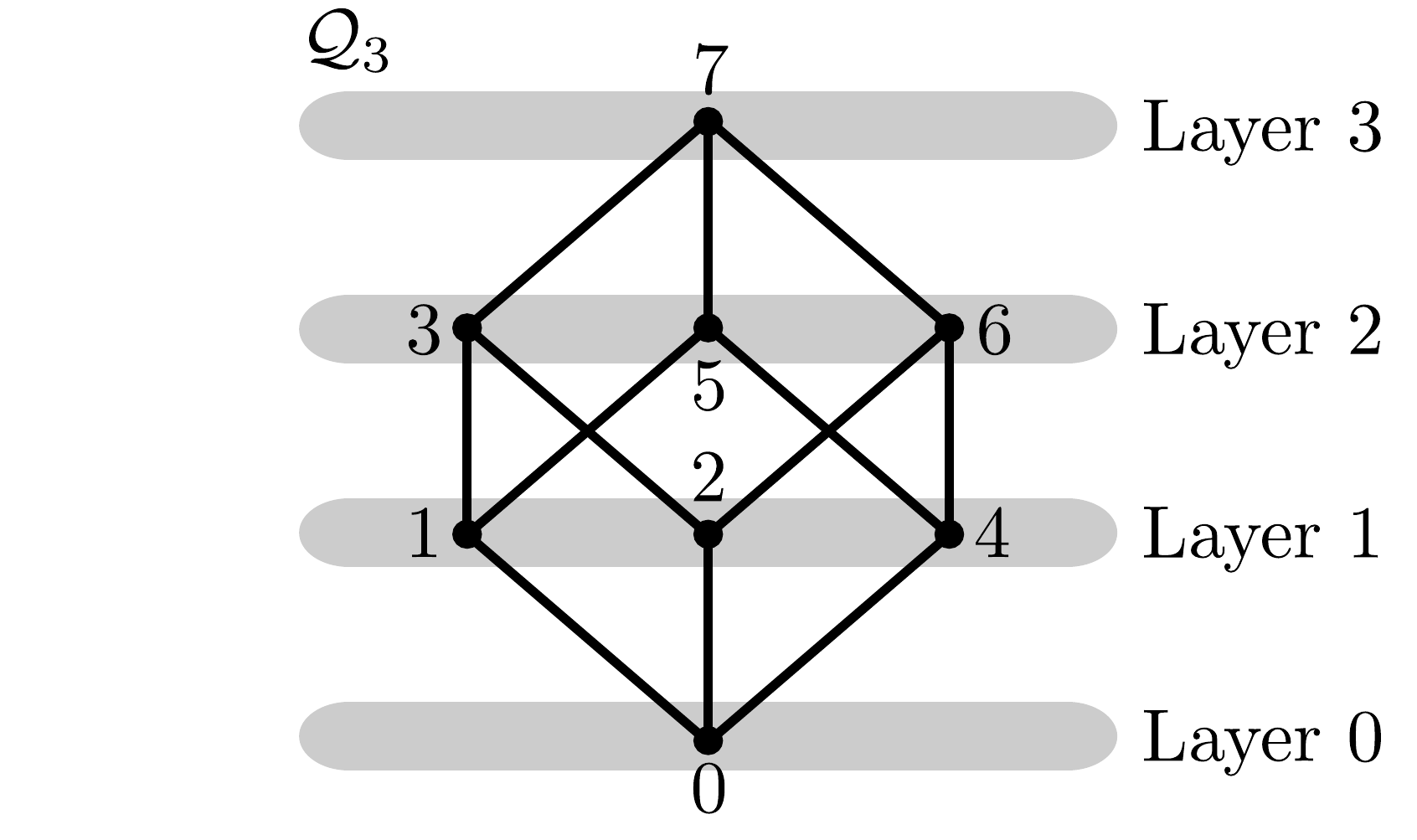}
\caption{A labelled $\mc{Q}_3$. The grey sets indicate the different
layers of the hypercube.}\label{Fig:Layers}
\end{center}
\end{figure}

We will consider two different types of Tur\'an problems involving
hypercubes. In the first type we will be interested in the following
question: given a forbidden family of graphs $\mc{F}$, what is the
maximum number of edges an $\mc{F}$-free subgraph of $\mc{Q}_n$ can
have? We are particularly interested in the limit of the maximum
hypercube edge density as $n$ tends to infinity, where we define the
\emph{hypercube edge density} of a subgraph $G$ of $\mc{Q}_n$ to be
$|E(G)|/|E(\mc{Q}_n)|$. We will refer to the limit as the \emph{edge
Tur\'an density}
\[
\pi_e(\mc{F}) =
\lim_{n\to\infty}\max_{G}\{|E(G)|/|E(\mc{Q}_n)|:G\subseteq\mc{Q}_n,\tr{
and is $\F$-free}\},
\]
a simple averaging argument shows it always exists.

Motivation to study the edge Tur\'an density comes from Erd\H os
\cite{EQ2} who conjectured that $\pi_e(\mc{Q}_2)=1/2$. It is easily
seen that $\pi_e(\mc{Q}_2)\geq 1/2$ by taking $\mc{Q}_n$ and
removing those edges that have one vertex in layer $2r-1$ and the
other in layer $2r$ for each $r$. Such subgraphs of $\mc{Q}_n$ are
$\mc{Q}_2$-free and contain exactly half the edges. The densest
known constructions which are $\mc{Q}_2$-free are given by Brass,
Harborth, and Nienborg \cite{BHN}, and they have a hypercube edge
density of approximately $(1+1/\sqrt{n})/2$. Chung \cite{Chung}
showed that $\pi_e(\mc{Q}_2)\leq (2+\sqrt{13})/9 = 0.62284$, her
argument was extended by Thomason and Wagner \cite{TW} using a
computer, to get the previously best known bound of $0.62256$. By
extending Razborov's semidefinite flag algebra technique to
hypercubes we will prove a significantly smaller upper bound of
$0.60680$. Later we will improve this bound further to $0.60318$
using partially defined hypercubes. Chung \cite{Chung} also
considered the edge Tur\'an density of $6$-cycles, and proved
$1/4\leq \pi_e(C_6) \leq \sqrt{2}-1 = 0.41421$. We will improve the
upper bound to $0.37550$, then using partially defined hypercubes
reduce it further to $0.36577$.

The second type of hypercube Tur\'an problem we will look at is very
similar to the first but focuses on the density of vertices rather
than edges. In the second type we are interested in the following
question: given a forbidden family of graphs $\mc{F}$, what is the
maximum number of vertices an $\mc{F}$-free induced subgraph of
$\mc{Q}_n$ can have? We are particularly interested in the limit of
the maximum hypercube vertex density as $n$ tends to infinity, where
we define the \emph{hypercube vertex density} of an induced subgraph
$G$ of $\mc{Q}_n$ to be $|V(G)|/|V(\mc{Q}_n)|$. We will refer to the
limit as the \emph{vertex Tur\'an density}
\[
\pi_v(\mc{F}) =
\lim_{n\to\infty}\max_{G}\{|V(G)|/|V(\mc{Q}_n)|:\tr{$G$ is an
$\F$-free induced subgraph of $\mc{Q}_n$}\},
\]
again a simple averaging argument shows it always exists.

The analogous problem to Erd\H os' conjecture is calculating
$\pi_v(\mc{Q}_2)$. E.A.\ Kostochka \cite{K} and independently
Johnson and Entringer \cite{JE} showed that $\pi_v(\mc{Q}_2)=2/3$.
Johnson and Talbot \cite{JT} proved that $\pi_v(R_1)=2/3$, where
$R_1$ is the graph formed by removing vertices $0$ and $1$ from
$\mc{Q}_3$. By extending Razborov's semidefinite flag algebra method
we will prove $\pi_v(R_2)=2/3$, where $R_2$ is the graph formed by
removing a single vertex from $\mc{Q}_3$. The value of
$\pi_v(\mc{Q}_3)$, however, still remains undetermined. A lower
bound of $3/4$ is easily achieved by considering the induced
subgraphs of $\mc{Q}_n$ formed by removing all vertices in layers
that are a multiple of four (i.e.\ layers $0,4,8,\ldots$). Although
we could not show $\pi_v(\mc{Q}_3)\leq 3/4$ we will prove
$\pi_v(\mc{Q}_3)\leq 0.76900$. We will also show that $1/2\leq
\pi_v(C_6)\leq 0.53111$.

\section{Vertex Tur\'an density}\label{SUBSEC:VertexHypercube}

Calculating the vertex Tur\'an density involves looking at induced
subgraphs of hypercubes. However, the structure of the hypercubes
may not be retained by the induced subgraphs. This structure will
prove to be useful and will simplify definitions later. Hence rather
than work directly with induced subgraphs we will instead use
vertex-coloured hypercubes that represent induced subgraphs of
$\mc{Q}_n$. In particular we will colour the vertices red and blue.
The induced subgraph that a red-blue vertex-coloured hypercube
represents can be constructed by removing those vertices that are
coloured red (as well as any edges that are incident to such
vertices) and keeping those vertices that are blue. The conjecture
that $\pi_v(\mc{Q}_3)=3/4$ comes from asking what is the maximum
number of vertices a $\mc{Q}_3$-free induced subgraph of $\mc{Q}_n$
can have. It is clear that this is equivalent to asking what is the
maximum number of blue vertices a vertex-coloured $\mc{Q}_n$ can
have such that it does not contain an all blue $\mc{Q}_3$. Therefore
the problem of calculating $\pi_v(\mc{Q}_3)$, and $\pi_v(\mc{F})$ in
general, can be translated into a problem involving forbidding
vertex-coloured hypercubes in a vertex-coloured $\mc{Q}_n$. We will
define the equivalent notion of vertex Tur\'an density for
vertex-coloured hypercubes shortly, but first we need some
definitions.

We will use the notation $(n,\kappa)_v$ to represent a
vertex-coloured $\mc{Q}_n$, where $\kappa:
V(\mc{Q}_n)\to\{\text{red, blue}\}$. We define $V(F)$ and $E(F)$ for
a vertex-coloured hypercube $F=(n,\kappa)_v$ to be $V(\mc{Q}_n)$ and
$E(\mc{Q}_n)$ respectively. Consider two vertex-coloured hypercubes
$F_1=(n_1,\kappa_1)_v,$ and $F_2=(n_2,\kappa_2)_v$. We say $F_1$ is
\emph{isomorphic} to $F_2$ if there exists a bijection $f:V(F_1)\to
V(F_2)$ such that for all $v_1v_2\in E(F_1)$, $f(v_1)f(v_2)\in
E(F_2)$ and for all $v\in V(F_1)$, $\kappa_1(v)=\kappa_2(f(v))$. We
say $F_1$ is a \emph{subcube} of $F_2$ if there exists an injection
$g:V(F_1)\to V(F_2)$ such that for all $v_1v_2\in E(F_1)$,
$g(v_1)g(v_2)\in E(F_2)$ and for all $v\in V(F_1)$ if
$\kappa_1(v)=\text{blue}$ then $\kappa_2(g(v))=\text{blue}$.

The \emph{vertex density} of $F=(n,\kappa)_v$ is
\[
d_v(F)=\frac{|\{v\in V(F) : \kappa(v)=\text{blue}\}|}{|V(F)|}.
\]
Note that this is analogous to the hypercube vertex density defined
in Section \ref{SUBSEC:IntroHypercube}. Given a family of
vertex-coloured hypercubes $\F$, we say $H$, a vertex-coloured
hypercube, is \emph{$\F$-free} if $H$ does not contain a subcube
isomorphic to any member of $\F$. The \emph{coloured vertex Tur\'an
density} of $\F$ is defined to be the following limit (a simple
averaging argument shows that it always exists)
\[
\pi_{cv}(\F)=\lim_{n\to\infty}
\max_\kappa\{d_v(H):H=(n,\kappa)_v\tr{ and is $\F$-free}\}.
\]

\begin{figure}[tbp]
\begin{center}
\includegraphics[height=3.6cm]{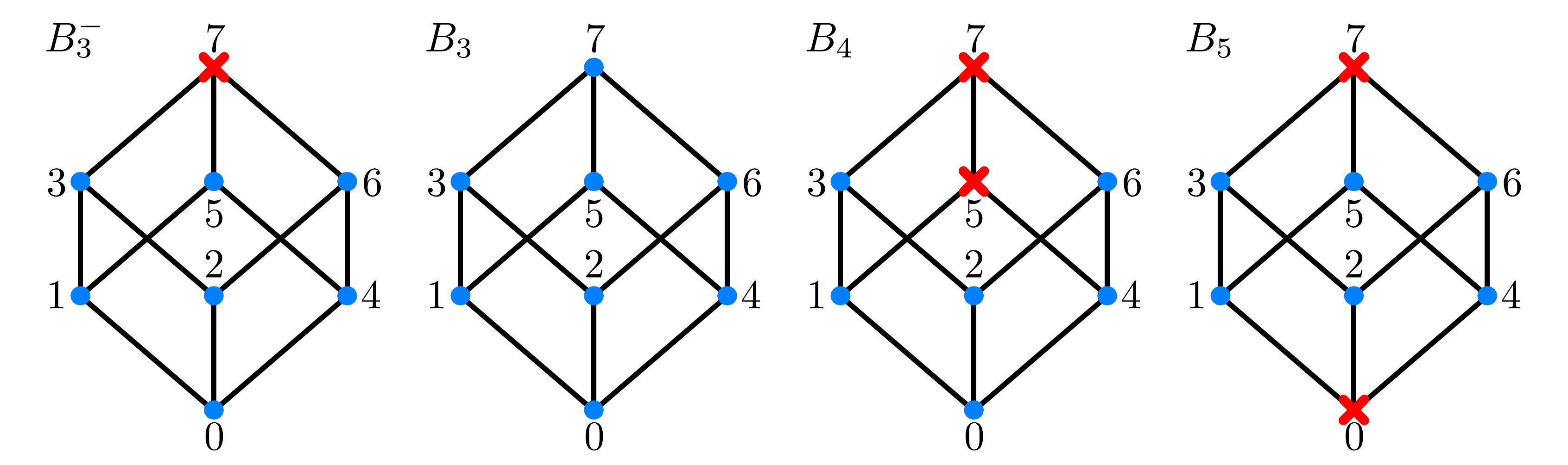}
\caption{The vertex-coloured hypercubes $B_3^-$, $B_3$, $B_4$, and
$B_5$. The blue vertices are represented by blue circles, and the
red vertices by red crosses.}\label{Fig:Vertex}
\end{center}
\end{figure}

Given these definitions it is easy to see that
$\pi_v(\mc{Q}_3)=\pi_{cv}(B_3)$, where $B_3$ is a $\mc{Q}_3$ with
all its vertices coloured blue, see Figure \ref{Fig:Vertex}. It is
also not hard to show that forbidding $R_2$ in $\mc{Q}_n$ is
equivalent to asking that a vertex-coloured hypercube is
$B_3^-$-free, where $B_3^-$ is a $\mc{Q}_3$ with vertex $7$ coloured
red and the remaining vertices coloured blue, see Figure
\ref{Fig:Vertex}. Hence $\pi_v(R_2)=\pi_{cv}(B_3^-)$. The final
vertex Tur\'an density we will consider is $\pi_v(C_6)$ which
requires a bit more work to convert into a vertex-coloured hypercube
problem. However, it is not too difficult to show that all
$6$-cycles in $\mc{Q}_n$ lie within a $\mc{Q}_3$ subgraph and it is
easy to check that there are only two distinct $6$-cycles in a
$\mc{Q}_3$ up to isomorphism. These two $6$-cycles can be
represented by two vertex-coloured hypercubes of dimension three
which we will call $B_4$ and $B_5$ (the names were chosen to be
consistent with \cite{B11}). Specifically vertices $5$ and $7$ are
coloured red in $B_4$ and vertices $0$ and $7$ are red in $B_5$ (the
remaining vertices are blue) see Figure \ref{Fig:Vertex}. Therefore
$\pi_v(C_6)=\pi_{cv}(B_4,B_5)$. By extending Razborov's method to
vertex-coloured hypercubes we will be able to prove the following
result.

\begin{theorem}\label{THM:CubeFree}
The following all hold:
\begin{enumerate}
\item[(i)] $\pi_v(R_2)=\pi_{cv}(B_3^-)=2/3$,
\item[(ii)] $3/4\leq \pi_v(\mc{Q}_3)=\pi_{cv}(B_3)\leq 0.76900$,
\item[(iii)] $1/2\leq \pi_v(C_6)=\pi_{cv}(B_4,B_5)\leq 0.53111$.
\end{enumerate}
\end{theorem}

The proofs of the lower bounds in Theorem \ref{THM:CubeFree} are
given by simple constructions. The lower bound of $2/3$ for
$\pi_v(R_2)$ is proved by considering the induced subgraph of
$\mc{Q}_n$ formed by removing every third layer of vertices.
Similarly $\pi_v(\mc{Q}_3)\geq 3/4$ and $\pi_v(C_6)\geq 1/2$ can be
proved by looking at $\mc{Q}_n$ with every fourth layer removed and
every second layer removed respectively.

The upper bounds in Theorem \ref{THM:CubeFree} were calculated using
Razborov's semidefinite flag algebra method the details of which are
given in the following section (Section \ref{SUBSEC:RazVertex}).
Specific data for the problems given in Theorem \ref{THM:CubeFree}
can be found in the data files \texttt{B3-.txt}, \texttt{B3.txt},
and \texttt{B4B5.txt}, see \cite{BV10}. The calculations required to
turn the data into upper bounds are too long to do by hand and so we
provide the program {\progCV} to verify our claims, see \cite{BV10}
(note that the program does not use floating point arithmetic so no
rounding errors can occur).

It is worth mentioning that there are many non-isomorphic extremal
constructions that are $B_3^-$-free and have an asymptotic vertex
density of $2/3$. As far as we are aware this is the first known
case where Razborov's method has given an exact upper bound when
there are multiple extremal constructions.

For completeness we will describe the known $B_3^-$-free extremal
constructions (though there may be more we are unaware of). To
create such a construction on $\mc{Q}_n$ first take any partition of
the vertices which divides $\mc{Q}_n$ into two disjoint
$\mc{Q}_{n-1}$ subcubes which we will call $S_1$ and $S_2$. Next for
each $S_i$ choose a canonical labelling of its vertices, and an
integer $z_i\in\{0,1,2\}$. The labelling we chose for each $S_i$
defines a layering of its vertices, which we will use to colour
them. In particular the vertices in layer $m$ of $S_i$ are coloured
red if $m\equiv z_i$ mod $3$ and blue otherwise. It is easy to check
that the resulting coloured $\mc{Q}_n$ is $B_3^-$-free and
asymptotically has $2/3$ of its vertices coloured blue.

\subsection{Razborov's method on vertex-coloured
hypercubes}\label{SUBSEC:RazVertex}

Let $\F$ be a family of coloured hypercubes whose coloured vertex
Tur\'an density we wish to compute (or at least approximate). Let
$\mc{H}$ be the family of all $\F$-free vertex-coloured hypercubes
of dimension $l$, up to isomorphism. If $l$ is sufficiently small we
can explicitly determine $\mc{H}$ (by computer search if necessary).
For $H\in \mc{H}$ and a large $\F$-free coloured hypercube $G$, we
define $p(H;G)$ to be the probability that a random hypercube of
dimension $l$ from $G$ induces a coloured subcube isomorphic to $H$.

Trivially, the vertex density of $G$ is equal to the probability
that a random $\mc{Q}_0$ (a single vertex) from $G$ is coloured
blue. Thus, averaging over hypercubes of dimension $l$ in $G$, we
can express the vertex density of $G$ as
\begin{align}\label{HypercubeTrivialDensity}
d_v(G) = \sum_{H\in \mc{H}}d_v(H)p(H;G),
\end{align}
and hence $\pi_{cv}(\mc{F})\leq \max_{H\in\mc{H}}d_v(H)$. This
``averaging'' bound can be improved upon by considering how small
pairs of $\mc{F}$-free hypercubes can intersect and how many times
these intersections appear in the $H$ hypercubes. To utilize this
information we will use Razborov's method and extend his notion of
flags and types to hypercubes.

For vertex-coloured hypercubes we define flags and types as follows.
A \emph{flag}, $F=(G_F,\theta)$, is a vertex-coloured hypercube
$G_F$ together with an injective map $\theta:
\{0,1,\ldots,2^s-1\}\to V(G_F)$ such that $\theta(i)\theta(j)\in
E(G_F)$ if and only if $i$ and $j$ differ by precisely one digit in
their binary representations (i.e.\ $\theta$ induces a canonically
labelled hypercube). If $\theta$ is bijective (and so
$|V(G_F)|=2^s$) we call the flag a \emph{type}. For ease of notation
given a flag $F=(G_F,\theta)$ we define its dimension $dim(F)$ to be
the dimension of the hypercube underlying $G_F$. Given a type
$\sigma$ we call a flag $F=(G_F,\theta)$ a $\sigma$-\emph{flag} if
the induced labelled and coloured subcube of $G_F$ given by $\theta$
is $\sigma$.

Fix a type $\sigma$ and an integer $m\leq (l+dim(\sigma))/2$. (The
bound on $m$ ensures that an $l$-dimensional hypercube can contain
two $m$-dimensional subcubes overlapping in a dimension
$dim(\sigma)$ hypercube.) Let $\msig$ be the set of all
$\mc{F}$-free $\sigma$-flags of dimension $m$, up to isomorphism.
Let $\Theta$ be the set of all injective functions from
$\{0,1,\ldots,2^{dim(\sigma)}-1\}$ to $V(G)$, that result in a
canonically labelled hypercube. Given $F\in\msig$ and
$\theta\in\Theta$ we define $p(F,\theta;G)$ to be the probability
that an $m$-dimensional coloured hypercube $R$ chosen uniformly at
random from $G$ subject to $\text{im}(\theta)\subseteq V(R)$,
induces a $\sigma$-flag $(R,\theta)$ that is isomorphic to $F$.

If $F_a,F_b\in \F_m^\sigma$ and $\theta\in \Theta$ then
$p(F_a,\theta;G)p(F_b,\theta;G)$ is the probability that two
$m$-dimensional coloured hypercubes $R_a,R_b$ chosen independently
at random from $G$ subject to $\text{im}(\theta)\subseteq V(R_a)\cap
V(R_b)$, induce $\sigma$-flags $(R_a,\theta)$, $(R_b,\theta)$ that
are isomorphic to $F_a, F_b$ respectively. We define the related
probability, $p(F_a,F_b,\theta;G)$, to be the probability that two
$m$-dimensional coloured hypercubes $R_a,R_b$ chosen independently
at random from $G$ subject to $\text{im}(\theta)=V(R_a)\cap V(R_b)$,
induce $\sigma$-flags $(R_a,\theta)$, $(R_b,\theta)$ that are
isomorphic to $F_a, F_b$ respectively. It is easy to show that
$p(F_a,\theta;G)p(F_b,\theta;G)=p(F_a,F_b,\theta;G)+o(1)$ where the
$o(1)$ term vanishes as $|V(G)|$ tends to infinity.

Consequently taking the expectation over a uniformly random choice
of $\theta\in \Theta$ gives
\begin{equation}\label{a1:eq}
\mathbf{E}_{\theta\in\Theta}\left[p(F_a,\theta;G)p(F_b,\theta;G)\right]
= \mathbf{E}_{\theta\in\Theta}\left[p(F_a,F_b,\theta;G)\right]+o(1).
\end{equation}
Furthermore the expectation on the right hand side of (\ref{a1:eq})
can be rewritten in terms of $p(H;G)$ by averaging over
$l$-dimensional hypercubes of $G$, hence
\begin{equation}\label{a2:eq}
\mathbf{E}_{\theta\in\Theta}\left[p(F_a,\theta;G)p(F_b,\theta;G)\right]
=
\sum_{H\in\mc{H}}\mathbf{E}_{\theta\in\Theta_H}\left[p(F_a,F_b,\theta;H)\right]p(H;G)+o(1),
\end{equation}
where $\Theta_H$ is the set of all injective maps
$\theta:\{0,1,\ldots,2^{dim(\sigma)}-1\}\to V(H)$ which induce a
canonically labelled hypercube. Note that the right hand side of
(\ref{a2:eq}) is a linear combination of $p(H;G)$ terms whose
coefficients can be explicitly calculated, this will prove useful
when used with (\ref{HypercubeTrivialDensity}) which is of a similar
form.

Given $\F_m^\sigma$ and a positive semidefinite matrix $Q=(q_{ab})$
of dimension $|\mc{F}^\sigma_m|$, let
$\mathbf{p}_\theta=(p(F,\theta;G):F\in \mc{F}^\sigma_m)$ for
$\theta\in\Theta$. Using (\ref{a2:eq}) and the linearity of
expectation we have
\begin{align}
0\leq
\mathbf{E}_{\theta\in\Theta}[\mathbf{p}_\theta^TQ\mathbf{p}_\theta]
&= \sum_{H\in\mc{H}} c_H p(H;G)+o(1) \label{q1:eq}
\end{align}
where
\[
c_H = \sum_{F_a,F_b\in \mc{F}^\sigma_m}q_{ab}\mathbf{E}_{\theta\in
\Theta_H}[p(F_a,F_b,\theta;H)].
\]
Note that $c_H$ is independent of $G$ and can be explicitly
calculated. Combining (\ref{q1:eq}) with
(\ref{HypercubeTrivialDensity}) allows us to write the following
\begin{align*}
d_v(G)\leq
d_v(G)+\mathbf{E}_{\theta\in\Theta}[\mathbf{p}_\theta^TQ\mathbf{p}_\theta]
= \sum_{H\in\mc{H}} (d_v(H)+c_H) p(H;G)+o(1).
\end{align*}
Hence
\[
\pi_{cv}(\F)\leq \max_{H\in\mc{H}}(d_v(H)+c_H).
\]
Note that some of the $c_H$ may be negative and so for a careful
choice of $Q$ this may result in a better bound for the Tur\'an
density than the simple ``averaging'' bound derived from
(\ref{HypercubeTrivialDensity}). Our task has therefore been reduced
to finding an optimal choice of $Q$ which will lower the bound as
much as possible. This is a convex optimization problem in
particular a semidefinite programming problem. As such we can use
freely available software such as CSDP \cite{B} to find $Q$ and
hence a bound on $\pi_{cv}(\F)$.

Our argument can also be extended to consider multiple types
$\sigma_i$, dimensions $m_i$, and positive semidefinite matrices
$Q_i$ (of dimension $|\F_{m_i}^{\sigma_i}|$), to create several
terms of the form
$\mathbf{E}_{\theta\in\Theta}[\mathbf{p}_{i,\theta}^TQ_i\mathbf{p}_{i,\theta}]$,
where $\mathbf{p}_{i,\theta}=(p(F,\theta;G):F\in
\mc{F}^{\sigma_i}_{m_i})$. By considering $d_v(G)+\sum_i
\mathbf{E}_{\theta\in\Theta}[\mathbf{p}_{i,\theta}^TQ_i\mathbf{p}_{i,\theta}]$
we can get a more complicated bound to optimize. However, it is
still expressible as a semidefinite program and the extra
information often results in a better bound.

It is worth noting that in order to get a tight bound for
$\pi_{cv}(B_3^-)$ we used the methods described in \cite{B11}
(Section 2.4.2) to remove the rounding errors. Although the method
described in \cite{B11} is for hypergraphs, due to its length and
the ease in to which it can be converted into the hypercube setting
we did not feel it was worth reproducing here.

\section{Edge Tur\'an density}\label{SUBSEC:EdgeHypercube}

In this section we will describe a relatively straightforward
extension of Razborov's method to the edge Tur\'an density problem
for hypercubes. Later we will give a more complicated extension
using partially defined hypercubes which results in improved bounds.

When we looked at the vertex Tur\'an density problem we found that
rather than working directly with subgraphs of hypercubes it was
simpler to use vertex-coloured hypercubes instead. Similarly when
calculating the edge Tur\'an density we will use red-blue
edge-coloured hypercubes to represent subgraphs of hypercubes. The
subgraph an edge-coloured hypercube represents can be constructed by
removing those edges that are coloured red and keeping those edges
that are blue.

We will use the notation $(n,\kappa)_e$ to represent an
edge-coloured $\mc{Q}_n$, where $\kappa: E(\mc{Q}_n)\to\{\text{red,
blue}\}$. We define $V(F)$ and $E(F)$ for an edge-coloured hypercube
$F=(n,\kappa)_e$ to be $V(\mc{Q}_n)$ and $E(\mc{Q}_n)$ respectively.
Consider two edge-coloured hypercubes $F_1=(n_1,\kappa_1)_e,$ and
$F_2=(n_2,\kappa_2)_e$. We say $F_1$ is \emph{isomorphic} to $F_2$
if there exists a bijection $f:V(F_1)\to V(F_2)$ such that for all
$v_1v_2\in E(F_1)$, $f(v_1)f(v_2)\in E(F_2)$ and
$\kappa_1(v_1v_2)=\kappa_2(f(v_1)f(v_2))$. We say $F_1$ is a
\emph{subcube} of $F_2$ if there exists an injection $g:V(F_1)\to
V(F_2)$ such that for all $v_1v_2\in E(F_1)$, $g(v_1)g(v_2)\in
E(F_2)$ and if $\kappa_1(v_1v_2)=\text{blue}$ then
$\kappa_2(g(v_1)g(v_2))=\text{blue}$.

The \emph{edge density} of $F=(n,\kappa)_e$ is
\[
d_e(F)=\frac{|\{v_1v_2\in E(F) :
\kappa(v_1v_2)=\text{blue}\}|}{|E(F)|}.
\]
Note that this is analogous to the hypercube edge density defined in
Section \ref{SUBSEC:IntroHypercube}. Given a family of coloured
hypercubes $\F$, we say $H$, a coloured hypercube, is
\emph{$\F$-free} if $H$ does not contain a subcube isomorphic to any
member of $\F$. The \emph{coloured edge Tur\'an density} of $\F$ is
defined to be the following limit (a simple averaging argument shows
that it always exists)
\[
\pi_{ce}(\F)=\lim_{n\to\infty}
\max_\kappa\{d_e(H):H=(n,\kappa)_e\tr{ and is $\F$-free}\}.
\]

\begin{figure}[tbp]
\begin{center}
\includegraphics[height=3.6cm]{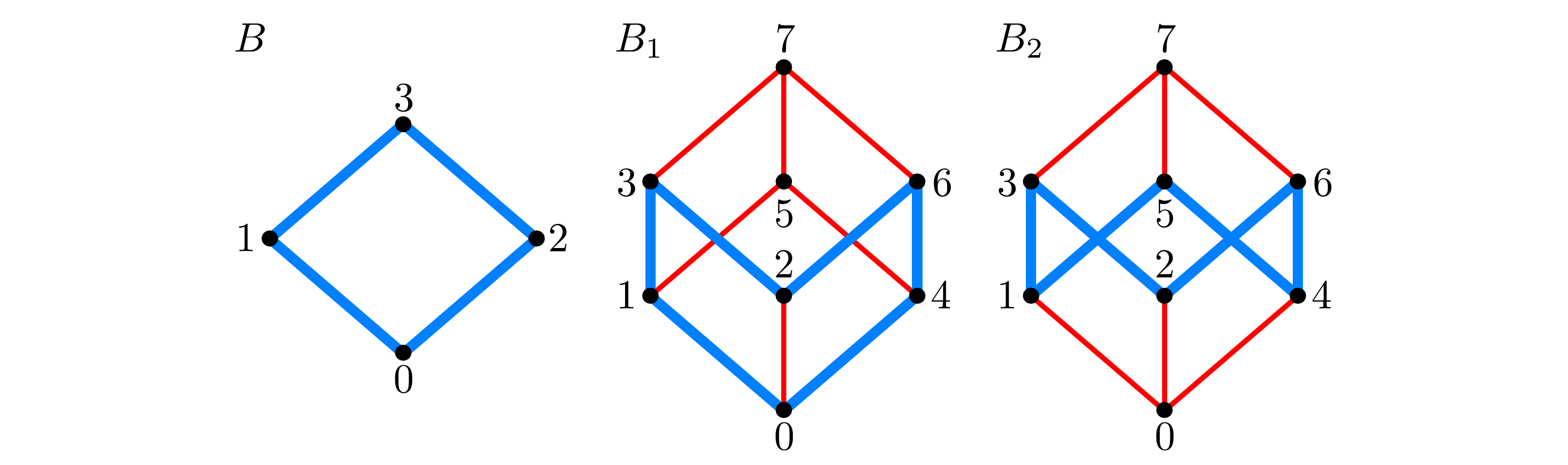}
\caption{The edge-coloured hypercubes $B$, $B_1$, and $B_2$. The
blue edges are represented by thick blue lines, and the red edges by
thin red lines.}\label{Fig:Edge}
\end{center}
\end{figure}

Given these definitions it is easy to see that
$\pi_e(\mc{Q}_2)=\pi_{ce}(B)$ where $B$ is a $\mc{Q}_2$ with all
four of its edges coloured blue, see Figure \ref{Fig:Edge}. We are
also interested in $\pi_e(C_6)$. It is not to difficult to show that
all $6$-cycles in $\mc{Q}_n$ lie within a $\mc{Q}_3$ subgraph. There
are two distinct $6$-cycles in a $\mc{Q}_3$ up to isomorphism, their
edge sets are $E_1 = \{01, 13, 32, 26, 64, 40\}$, and $E_2 = \{51,
13, 32, 26, 64, 45\}$. Let $B_1$ be a $\mc{Q}_3$ with those edges in
$E_1$ coloured blue and the remaining edges coloured red, see Figure
\ref{Fig:Edge}. Similarly let $B_2$ be a $\mc{Q}_3$ with those edges
in $E_2$ coloured blue and the remaining edges coloured red, see
Figure \ref{Fig:Edge}. Hence forbidding a blue edged $6$-cycle in an
edge-coloured hypercube is equivalent to requiring that it is $B_1$
and $B_2$-free. Therefore $\pi_e(C_6)=\pi_{ce}(B_1, B_2)$. By
extending Razborov's semidefinite flag algebra method to
edge-coloured hypercubes, we will be able to prove the following
bounds on $\pi_{ce}(B)$ and $\pi_{ce}(B_1,B_2)$.

\begin{theorem}\label{THM:SquareFree}
$\pi_e(\mc{Q}_2)=\pi_{ce}(B)\leq 0.60680$ and
$\pi_e(C_6)=\pi_{ce}(B_1, B_2)\leq 0.37550$.
\end{theorem}

We omit the details of extending Razborov's technique to
edge-coloured hypercubes; it is virtually identical to the extension
described in Section \ref{SUBSEC:RazVertex}, with the term
``edge-coloured'' replacing the term ``vertex-coloured''.

\begin{proof}[Proof of Theorem \ref{THM:SquareFree}]
All the necessary data (types, flags, matrices, etc.) needed to
prove $\pi_{ce}(B)\leq 0.60680$ can be found in the file
\texttt{B.txt} \cite{BE10}. The calculation that converts this data
into an upper bound is too long to do by hand and so we provide the
program {\progCE} (see \cite{BE10}) to verify our claim (the program
does not use floating point arithmetic so no rounding errors can
occur). Similarly the data needed to prove $\pi_{ce}(B_1,B_2)\leq
0.37550$ can be found in the file \texttt{B1B2.txt} \cite{BE10}.
\end{proof}

\section{Partially defined hypercubes}\label{SUBSEC:PartialHypercubes}

In this section we improve the bounds given in Theorem
\ref{THM:SquareFree} by applying a slightly modified version of
Razborov's method to partially defined hypercubes. We define a
\emph{partially defined hypercube} simply to be an edge-coloured
hypercube where instead of colouring the edges with just two
colours, red and blue, we use three colours, red, blue, and grey.
Note that throughout this section we will use the less cumbersome
term \emph{red-blue hypercube} to refer to red-blue edge-coloured
hypercubes (as defined in Section \ref{SUBSEC:EdgeHypercube}).

The interpretation we gave to a red-blue hypercube in Section
\ref{SUBSEC:EdgeHypercube} was that it represented a subgraph whilst
retaining the underlying structure of the hypercube. The subgraph
could be reconstructed by removing those edges which are red, and
keeping those which are blue. We have a similar interpretation for
partially defined hypercubes, as before those edges that are red or
blue represent edges to remove or keep respectively, but grey edges
represent edges which are undefined (i.e.\ the colouring does not
specify whether to remove the edge or not). Hence a partially
defined hypercube does not represent a single subgraph but a set of
subgraphs.

The use of grey edges causes us to lose some information that could
be useful in bounding the Tur\'an density, but it also reduces the
size of the computations which we can use to our advantage. To
illustrate this point further we first note that in order to
calculate the upper bound for $\pi_{ce}(B)$ using Razborov's method
we need to explicitly determine $\mc{H}$ the family of all
$l$-dimensional $B$-free red-blue hypercubes for some choice of $l$.
The size of $\mc{H}$ gives us a rough indication of how hard the
computation will be. To get the bound given in Theorem
\ref{THM:SquareFree} we looked at $3$-dimensional red-blue
hypercubes (i.e.\ $l=3$) which results in $|\mc{H}| = 99$, and
consequently the computation is very quick on most computers. We can
achieve a better bound by looking at $4$-dimensional red-blue
hypercubes but this unfortunately results in $|\mc{H}|=3212821$
which is currently computationally unfeasible on an average
computer. However, by looking at $4$-dimensional partially defined
hypercubes we can reduce $|\mc{H}|$ to the more feasible value of
$90179$ and still make use of some of the information held in
red-blue hypercubes of dimension $4$. In fact we have some choice
over how many grey edges our partially defined hypercubes will
contain which translates into some control over how large we wish to
make $|\mc{H}|$ and how difficult a computation we want to attempt.

\begin{figure}[tbp]
\begin{center}
\includegraphics[height=3.6cm]{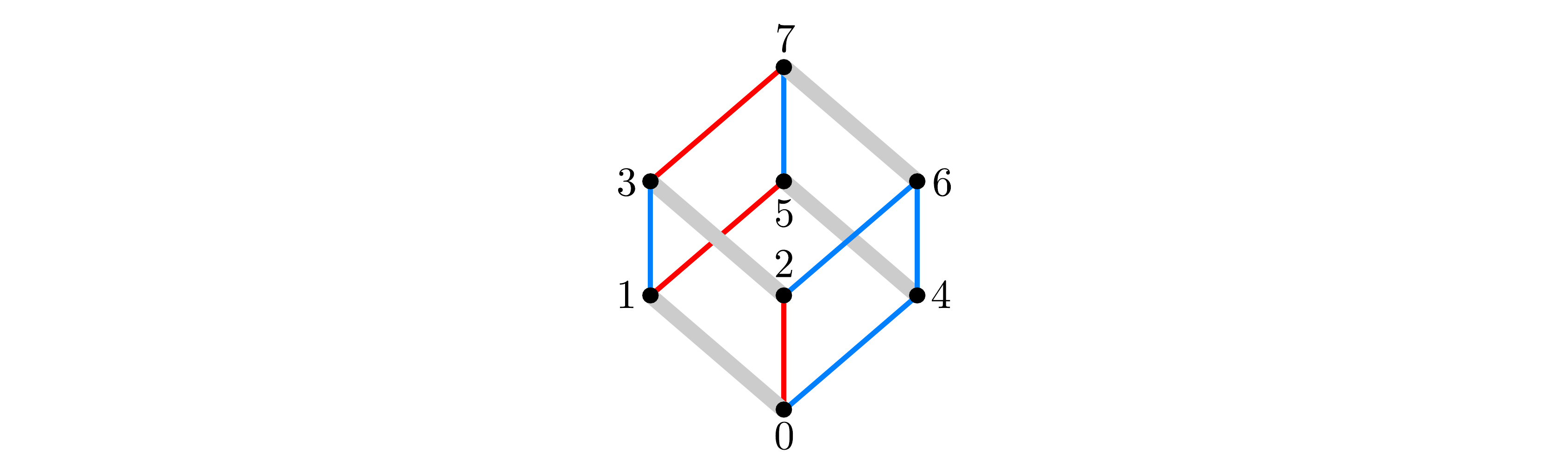}
\caption{An example of a partial hypercube. The grey coloured edges
are $01, 23, 45,$ and $67$, the other edges are coloured red and
blue.}\label{Fig:Partial}
\end{center}
\end{figure}

In this section we are primarily going to consider a very specific
type of partially defined hypercube, namely ones where if $ij$ is an
edge then it is grey if and only if $|i-j|=1$, see Figure
\ref{Fig:Partial} for an example.
%Under this restriction all $n$-dimensional partial
%hypercubes will look like two disjoint $(n-1)$-dimensional red-blue
%hypercubes with grey edges going between them.
We will refer to such a partially defined hypercube simply as a
\emph{partial hypercube}.

By applying a slightly modified version of Razborov's method to
partial hypercubes we can improve the bounds given in Theorem
\ref{THM:SquareFree}.

\begin{theorem}\label{THM:SquareFree_Partial}
$\pi_e(\mc{Q}_2)=\pi_{ce}(B)\leq 0.60318$ and
$\pi_e(C_6)=\pi_{ce}(B_1, B_2)\leq 0.36577$.
\end{theorem}

The relevant data files \texttt{PartialB.txt},
\texttt{PartialB1B2.txt}, and the program {\progCPE} used for
verification purposes can be found in the source files section on
the arXiv see \cite{BP}.

The rest of this section will be devoted to explaining the technical
details of extending flag algebras to partial hypercubes. We will
extend Razborov's method to partial hypercubes in the simplest and
most obvious way in Section \ref{SUBSEC:RazPartial}. Unfortunately
this extension does not produce better bounds than those given in
Theorem \ref{THM:SquareFree}. We will remedy this in Section
\ref{SUBSEC:PartialExtra} by incorporating some extra constraints
into the method.

\subsection{Razborov's method on partial hypercubes}\label{SUBSEC:RazPartial}

We will begin with some basic definitions. Note that the definitions
and explanations will involve red-blue hypercubes as defined in
Section \ref{SUBSEC:EdgeHypercube} so care must be taken to avoid
confusion.

We will use the notation $(n,\kappa)_p$ to formally represent an
$n$-dimensional partial hypercube, where $\kappa:
E(\mc{Q}_n)\to\{\text{red, blue, grey}\}$, and
$\kappa(v_1v_2)=\text{grey}$ if and only if $|v_1-v_2|=1$. We define
$V(F)$ and $E(F)$ for a partial hypercube $F=(n,\kappa)_p$ to be
$V(\mc{Q}_n)$ and $E(\mc{Q}_n)$ respectively. Consider two partial
hypercubes $F_1=(n_1,\kappa_1)_p,$ and $F_2=(n_2,\kappa_2)_p$. We
say $F_1$ is \emph{isomorphic} to $F_2$ if there exists a bijection
$f:V(F_1)\to V(F_2)$ such that for all $v_1v_2\in E(F_1)$,
$f(v_1)f(v_2)\in E(F_2)$ and
$\kappa_1(v_1v_2)=\kappa_2(f(v_1)f(v_2))$.
%We say $F_1$ is a \emph{subcube} of $F_2$ if there exists an
%injection $g:V(F_1)\to V(F_2)$ such that for all $v_1v_2\in E(F_1)$,
%$g(v_1)g(v_2)\in E(F_2)$ and if $\kappa_1(v_1v_2)=\text{blue}$ then
%$\kappa_2(g(v_1)g(v_2))=\text{blue}$.
The \emph{edge density} of $F=(n,\kappa)_p$ is
\[
d_p(F)=\frac{|\{v_1v_2\in E(F) :
\kappa(v_1v_2)=\text{blue}\}|}{|\{v_1v_2\in E(F) :
\kappa(v_1v_2)\neq\text{grey}\}|}.
\]
To ease notation later we define $\mc{P}$ to be a function which
converts red-blue hypercubes to partial hypercubes by colouring some
edges grey. In particular $\mc{P}((n,\kappa)_e)=(n,\kappa')_p$ where
$\kappa'(v_1v_2)=\text{grey}$ if $|v_1-v_2|=1$ and
$\kappa'(v_1v_2)=\kappa(v_1v_2)$ otherwise.

Let $\F$ be the family of red-blue hypercubes, for which we are
trying to compute an upper bound of $\pi_{ce}(\F)$.

In the red-blue hypercube case our application of Razborov's method
involved computing $\mc{H}$ the family of all $\F$-free red-blue
hypercubes of dimension $l$ up to isomorphism (for some choice of
$l$). In the partial hypercube version of Razborov's method,
$\mc{H}$ will be a family of partial hypercubes, and we will require
that it retain the key property that it represents all the possible
$l$-dimensional subcubes that could appear in a large $\F$-free
red-blue hypercube. With this in mind we make the following
definitions. We say a partial hypercube $(n,\kappa)_p$ is
\emph{$\F$-free} if $(n,\kappa')_e$ is $\F$-free where
$\kappa'(v_1v_2)=\text{red}$ if $\kappa(v_1v_2)=\text{grey}$ and
$\kappa'(v_1v_2)=\kappa(v_1v_2)$ otherwise. Using this we define
$\mc{H}$ to be the family of $\F$-free partial hypercubes of
dimension $l$, up to isomorphism.

For $H\in \mc{H}$ and a large $\F$-free red-blue hypercube $G$, we
define $p(H;G)$ to be the probability that a random hypercube of
dimension $l$ from $G$, together with a random canonical labelling
of its vertices, induces a coloured red-blue subcube $H'$ such that
$\mc{P}(H')$ is isomorphic to $H$. Note that
$\sum_{H\in\mc{H}}p(H;G)=1$. Trivially, the edge density of $G$ is
equal to the probability that a random edge from $G$ is coloured
blue. Thus, averaging over partial hypercubes of dimension $l$ in
$G$, we can express the edge density of $G$ as
\begin{align}\label{Eq:PartialTrivialDensity}
d_e(G) = \sum_{H\in \mc{H}}d_p(H)p(H;G),
\end{align}
(provided $l\geq 2$) and hence $\pi_{ce}(\mc{F})\leq
\max_{H\in\mc{H}}d_p(H)$. This bound can be improved upon by
considering how small pairs of $\mc{F}$-free partial hypercubes can
intersect, we use Razborov's notion of flags and types to describe
and utilize this information.

For partial hypercubes we define flags and types as follows. A
\emph{flag}, $F=(G_F,\theta)$, is a partial hypercube
$G_F=(n,\kappa)_p$ for some $n\geq 1$ and $\kappa$, together with an
injective map $\theta: \{0,1,\ldots,2^s-1\}\to V(G_F)$ for some
$s\geq 1$, such that $\theta(i)\theta(j)\in E(G_F)$ if and only if
$i$ and $j$ differ by precisely one digit in their binary
representations and $\kappa(\theta(i)\theta(j))=\text{grey}$ when
$|i-j|=1$ (i.e.\ $\theta$ induces a canonically labelled partial
hypercube). If $\theta$ is bijective (and so $|V(G_F)|=2^s$) we call
the flag a \emph{type}. For ease of notation given a flag
$F=(G_F,\theta)$ we define its dimension $dim(F)$ to be the
dimension of the hypercube underlying $G_F$. Given a type $\sigma$
we call a flag $F=(G_F,\theta)$ a $\sigma$-\emph{flag} if the
induced labelled partial subcube of $G_F$ given by $\theta$ is
$\sigma$.

We define $\msig$ be the set of all $\mc{F}$-free $\sigma$-flags of
dimension $m$, up to isomorphism (where $m\leq (l+dim(\sigma))/2$).
Let $\Theta$ be the set of all injective functions from
$\{0,1,\ldots,2^{dim(\sigma)}-1\}$ to $V(G)$, that result in a
canonically labelled hypercube. Given $F\in\msig$ and
$\theta\in\Theta$ we define $p(F,\theta;G)$ to be the probability
that an $m$-dimensional hypercube $R$ chosen uniformly at random
from $G$ subject to $\text{im}(\theta)\subseteq V(R)$ induces a
$\sigma$-flag with labelled vertices given by $\theta$ that is
isomorphic to $F$.

Given $\mathbf{p}_\theta=(p(F,\theta;G):F\in \mc{F}^\sigma_m)$, and
$Q$ a positive semidefinite matrix, we can show
$\mathbf{E}_{\theta\in\Theta}[\mathbf{p}_\theta^TQ\mathbf{p}_\theta]$
is non-negative and can be written explicitly in terms of $p(H;G)$
for $H\in\mc{H}$. We omit the details as the argument is virtually
identical to that given in Section \ref{SUBSEC:RazVertex}. As before
we can consider multiple types $\sigma_i$, and positive semidefinite
matrices $Q_i$ to create multiple terms of the form
$\mathbf{E}_{\theta\in\Theta}[\mathbf{p}_{i,\theta}^TQ_i\mathbf{p}_{i,\theta}]$
which will result in better bounds on $\pi_{ce}(\F)$. Let
\begin{align}\label{Eq:cH_Partial}
\sum_i{\mathbf{E}_{\theta\in\Theta}[\mathbf{p}_{i,\theta}^TQ_i\mathbf{p}_{i,\theta}]}
= \sum_{H\in\mc{H}}c_Hp(H;G)+o(1),
\end{align}
where $c_H$ can be explicitly calculated from $H$, and the $o(1)$
term vanishes as $|V(G)|$ tends to infinity. Using
(\ref{Eq:cH_Partial}) together with the fact that the $Q_i$ are
positive semidefinite, and applying it to
(\ref{Eq:PartialTrivialDensity}) gives us
\begin{align}\label{Eq:PartialNewD}
d_e(G)\leq
d_e(G)+\sum_i{\mathbf{E}_{\theta\in\Theta}[\mathbf{p}_{i,\theta}^TQ_i\mathbf{p}_{i,\theta}]}
= \sum_{H\in\mc{H}}(d_p(H)+c_H)p(H;G)+o(1).
\end{align}
Therefore
\[
\pi_{ce}(\F)\leq \max_{H\in\mc{H}}(d_p(H)+c_H).
\]
The matrices $Q_i$ and hence the optimal bound on $\pi_{ce}(\F)$ can
be determined by solving a semidefinite program.

\subsection{Additional constraints}\label{SUBSEC:PartialExtra}

Unfortunately the method as it stands does not perform well.
Considering $l$-dimensional $\F$-free partial hypercubes produces
precisely the same bound as considering $(l-1)$-dimensional red-blue
hypercubes and the latter results in a significantly easier
computation. The reason for this lack of improvement is that the
semidefinite program we create does not encode that we are looking
at $l$-dimensional partial hypercubes. We get the same semidefinite
program by instead looking at any two $(l-1)$-dimensional red-blue
subcubes of $G$ artificially made into a partial hypercube by adding
grey edges between them. To be clear, in this case we do not require
that the two $(l-1)$-dimensional red-blue hypercubes come from the
same $l$-dimensional hypercube (whereas in the partial hypercube
case we do). Consequently we gain no extra information that was not
already given in $(l-1)$-dimensional red-blue hypercubes, and hence
no improvement in the bound.

We remedy this situation by introducing some linear constraints of
the form $\sum_{H\in\mc{H}} a_H p(H;G)=0$ which hold for all
$\F$-free $G$, and the coefficients $a_H$ can be explicitly
calculated from the partial hypercubes $H\in\mc{H}$. Given a set of
such constraints $\sum_{H\in\mc{H}} a_{j,H} p(H;G)=0$, indexed by
$j$, and an associated set of real values $\mu_j$ we have
\begin{align*}
0 = \sum_j \mu_j \sum_{H\in\mc{H}}a_{j,H}p(H;G) =
\sum_{H\in\mc{H}}\alpha_H p(H;G)
\end{align*}
where $\alpha_H = \sum_j \mu_j a_{j,H}$. Applying this to
(\ref{Eq:PartialNewD}) gives us
\[
d_e(G)\leq \sum_{H\in\mc{H}} (d_p(H)+c_H+\alpha_H) p(H;G)+o(1),
\]
which implies a potentially better bound of
\[
\pi_{ce}(\F)\leq \max_{H\in\mc{H}}(d_p(H)+c_H+\alpha_H).
\]
Note that we can still pose the question of determining the optimal
semidefinite matrices and coefficients $\mu_j$ as a single
semidefinite program.

The linear constraints we will use come from calculating the
probabilities of various partially defined hypercubes appearing in
$G$. Let us define $\mc{S}$ to be the set of all red-blue-grey
edge-coloured $\mc{Q}_l$ (where $l$ is the dimension of the
$H\in\mc{H}$) with the property that $ij\in E(\mc{Q}_l)$ is coloured
grey if and only if $|i-j|=1$ or $2$ (note that such partially
defined hypercubes are not partial hypercubes). We will show for
each $S\in\mc{S}$ we can form a linear constraint of the required
form $\sum a_H p(H;G)=0$, but first we require some definitions.
%(though some of these constraints may be redundant as they may have $a_H=0$ for all $H\in\mc{H}$).

We say an edge $ij\in E(\mc{Q}_l)$ \emph{lies across dimension} $d$
if $|i-j|=2^d$. Define a \emph{$\mc{Q}_l$-mapping} to be a bijective
function $\phi:V(\mc{Q}_l)\to V(\mc{Q}_l)$ satisfying $ij\in
E(\mc{Q}_l)$ if and only if $\phi(i)\phi(j)\in E(\mc{Q}_l)$. Given
$a,b\in\{0,\ldots,l-1\}$ define $\phi_{ab}$ to be the
$\mc{Q}_l$-mapping where the binary representation of $\phi_{ab}(i)$
is the same as that of $i$ but with bits $a$ and $b$ swapped (we
take bit $0$ to be the least significant bit). Note that the effect
of $\phi_{ab}$ is to swap the edges that lie across dimension $a$
with those that lie across dimension $b$. Let $\Phi$ be the set of
all $\mc{Q}_l$-mappings $\phi$ that satisfy for all $ij\in
E(\mc{Q}_l)$ with $|i-j|=1$ or $2$, $|i-j|=|\phi(i)-\phi(j)|$. In
other words $\Phi$ consists of all $\mc{Q}_l$-mappings which keeps
those edges lying across dimension $0$ or $1$ still lying across
dimension $0$ or $1$ respectively. Given two $l$-dimensional
partially defined hypercubes $F_1,F_2$, we say $F_1$ is isomorphic
to $F_2$ under $\Phi$ if there exists $\phi\in\Phi$ such that the
colour of every edge $ij$ in $F_1$ matches that of $\phi(i)\phi(j)$
in $F_2$. Recall that in Section \ref{SUBSEC:RazPartial} we defined
the function $\mc{P}$ to convert red-blue hypercubes into partial
hypercubes by recolouring the edges that lie across dimension $0$ to
grey. We similarly define $\mc{P}'$ to be a function which recolours
the edges that lie across dimension $0$ or $1$ to grey. Finally,
given a $\mc{Q}_l$-mapping $\phi$ and a partially defined hypercube
$F$ of dimension $l$, we will abuse notation and take $\phi(F)$ to
be a partially defined hypercube of dimension $l$ with edge
$\phi(i)\phi(j)$ in $\phi(F)$ having the same colour as $ij$ in $F$.

%Define a \emph{$\mc{Q}_l$-mapping} to be a bijective function
%$\phi:V(\mc{Q}_l)\to V(\mc{Q}_l)$ satisfying $ij\in E(\mc{Q}_l)$ if
%and only if $\phi(i)\phi(j)\in E(\mc{Q}_l)$. Given
%$a,b\in\{0,\ldots,l-1\}$ define $\phi_{ab}$ to be the
%$\mc{Q}_l$-mapping where the binary representation of $\phi_{ab}(i)$
%is the same as that of $i$ but with bits $a$ and $b$ swapped (we
%take bit $0$ to be the least significant bit). Let $\Phi$ be the set
%of all $\mc{Q}_l$-mappings $\phi$ that satisfy for all $ij\in
%E(\mc{Q}_l)$ with $|i-j|=1$ or $2$, $|i-j|=|\phi(i)-\phi(j)|$. Given
%two $l$-dimensional partially defined hypercubes $F_1,F_2$, we say
%$F_1$ is isomorphic to $F_2$ under $\Phi$ if there exists
%$\phi\in\Phi$ such that the colour of every edge $ij$ in $F_1$
%matches that of $\phi(i)\phi(j)$ in $F_2$. Recall that in Section
%\ref{SUBSEC:RazPartial} we defined the function $\mc{P}$ to convert
%red-blue hypercubes into partial hypercubes by recolouring edges
%$ij$ grey if $|i-j|=1$. We similarly define $\mc{P}'$ to be a
%function which recolours edges $ij$ grey if $|i-j|=1$ or $2$.
%Finally, given a $Q_l$-mapping $\phi$ and a partially defined
%hypercube $F$ of dimension $l$, we will abuse notation and take
%$\phi(F)$ to be a partially defined hypercube of dimension $l$ with
%edge $\phi(i)\phi(j)$ in $\phi(F)$ having the same colour as $ij$ in
%$F$.

Given $S\in\mc{S}$ we will form our linear constraint by considering
$p_{\Phi}(S;G)$ the probability that a random hypercube of dimension
$l$ from $G$, together with a random canonical labelling of its
vertices, induces a coloured red-blue subcube $S'$ such that
$\mc{P}'(S')$ is isomorphic to $S$ under $\Phi$. We can calculate
$p_{\Phi}(S;G)$ explicitly in terms of $p(H;G)$ by taking $S'$ then
applying $\mc{P}$ to get a partial hypercube $H\in\mc{H}$. With
probability $p(H;G)$, $\mc{P}(S')$ is isomorphic to $H$, so we can
write
\[
p_{\Phi}(S;G) = \sum_{H\in\mc{H}}p_{\Phi}(S;H)p(H;G)
\]
where $p_{\Phi}(S;H)$ is the probability that
$\mc{P}'(\phi_{1n}(H))$ is isomorphic to $S$ under $\Phi$ for some
random choice of $n$ from $\{1,\ldots,l-1\}$ (we interpret
$\phi_{11}$ as the identity map).

It should be clear that since $p_{\Phi}(S;G)$ involves looking at a
random red-blue hypercube, that $p_{\Phi}(S;G) =
p_{\Phi}(\phi_{01}(S);G)$ and both the left and right hand sides can
be written in terms of $p(H;G)$ giving us the linear constraint
\[
\sum_{H\in\mc{H}}\big(p_{\Phi}(S;H)-p_{\Phi}(\phi_{01}(S);H)\big)p(H;G)
= 0
\]
as desired.

\section{Other partially defined objects}

In Section \ref{SUBSEC:PartialHypercubes} we gave a concrete example
of using partially defined graphs to improve a bound. In particular
the graph was a hypercube and we chose our undefined edges in a very
particular way. It is worth noting that there are many other ways we
could have chosen the undefined edges, for example we could have
said that an edge $v_1v_2$ is grey if and only if both $v_1$ and
$v_2$ are odd. Alternatively we could have taken $\mc{H}$ to be all
red-blue-grey edge-coloured hypercubes with precisely $t$ grey edges
for some fixed choice of $t$. This choice of how we choose the
undefined edges gives us some control of the size of the computation
and the improvement in the bound of the Tur\'an density we can
expect. We by no means think that Theorem
\ref{THM:SquareFree_Partial} gives the best possible bounds using
partial hypercubes, however the description of partial hypercubes we
used gave us a reasonable improvement and applying Razborov's method
to it was not too difficult.

Although so far we have only looked at Tur\'an density problems on
hypercubes it should be clear that the method of using partially
defined objects will work in most of the applications Razborov's
method has been applied. As an example we will briefly describe one
way it can be used in the case of $3$-uniform hypergraphs, or
$3$-graphs for short.

\subsection{Partially defined $3$-graphs}

We will pick our undefined edges in the following way, choose two
distinct vertices $u$, $v$ from the $3$-graph and make every triple
containing precisely one of $\{u,v\}$ an undefined edge. Such
partially defined graphs have the nice property that if we remove
$u$, and $v$ we get a fully defined standard $3$-graph the only
information we have lost is from edges and non-edges of the form
$uvw$, which we can recover by colouring vertex $w$ red and blue say
to represent whether $uvw$ was a non-edge or an edge respectively.
Hence we can represent our partially defined $3$-graphs as red-blue
vertex-coloured $3$-graphs. Note that linear constraints such as
those given in Section \ref{SUBSEC:PartialExtra} are still needed to
encode that the coloured vertices represent edges and non-edges,
these can be obtained by again calculating the probabilities of
partially defined graphs appearing in two different ways.

We have used such partially defined graphs to improve the bound of
$\pi(K_4^3)$, the Tur\'an density of the complete $3$-graph on $4$
vertices. The best known bound was held by Razborov \cite{R4} at
$0.56167$ by considering $3$-graphs of order $6$. We can decrease
this to $0.5615$ by looking at red-blue vertex-coloured $3$-graphs
of order $6$, together with regularity constraints as described by
Hladk\'y, Kr\'al', and Norine \cite{HKN}. We will describe the
regularity constraints in more detail in Section
\ref{SUBSEC:Regularity}. The relevant data required to prove the
$0.5615$ bound can be found in \texttt{K4.txt} located in the source
files section on the arXiv, see \cite{BP}.

Although the amount we have decreased the bound by is not that
impressive, a significantly better bound may be possible, as due to
time restrictions we did not make full use of the information
contained in the flags. In particular we chose to ignore the colours
of the non-labelled vertices. This was achieved by redefining a flag
to be a (partially labelled) red-blue-grey vertex-coloured
$3$-graph, where the non-labelled vertices are coloured grey to
represent undefined edges.

It is important to note that there are many alternative types of
partially defined $3$-graphs we could try instead of red-blue vertex
coloured $3$-graphs, any one of which may result in a significant
improvement to the bound of $\pi(K_4^3)$. Our choice to use red-blue
vertex-coloured $3$-graphs was motivated only by the fact that it
would be simple to explain and implement.

\subsubsection{Regularity constraints}\label{SUBSEC:Regularity}

Our proof that $\pi(K_4^3)\leq 0.5615$ involves regularity
constraints such as those described by Hladk\'y, Kr\'al', and Norine
for digraphs \cite{HKN}. In this section we will describe the
constraints and show they can be applied to the problem of bounding
$\pi(\F)$ for any forbidden family of covering $3$-graphs $\F$. A
graph is said to be \emph{covering} if every pair of vertices belong
to an edge.

To get a bound for $\pi(\F)$ Razborov's method involves looking at a
large $\F$-free graph $G$ to get a bound for the density $d(G)$. By
considering a sequence of edge maximal $\F$-free graphs of
increasing order, $d(G)$ can be made to tend to $\pi(\F)$ thereby
giving us a bound on the Tur\'an density. We will instead work with
a sequence of $\F$-free graphs $\{G_n\}_{n=1}^\infty$, which have
the property that $|V(G_n)| = n$, and $(1-o(1))n$ of the vertices of
$G_n$ have a degree density which is at most $o(1)$ from $\rho$ for
some fixed value $\rho$, where $o(1)$ tends to $0$ as $n$ tends to
infinity. (We define the \emph{degree density} of a vertex $v$ to be
the number of edges containing $v$ divided by the number of triples
containing $v$.) We will begin by showing that for any non-negative
$\rho\leq\pi(\F)$ there exists just such a sequence of graphs. Hence
when we apply Razborov's method to $G_n$ we can not only assume it
is $\F$-free but that it is also almost regular (in a very precise
sense) which is a condition we can use to achieve a better upper
bound.

We will create $\{G_n\}_{n=1}^\infty$ from $\{M_n\}_{n=1}^\infty$ a
sequence of extremal $\F$-free graphs with $|V(M_n)| = n$ and
$|E(M_n)| = ex(n,\F)$. The difference between the maximum and
minimum degree density in $M_n$ can be at most $2/(n-1)$ otherwise
removing the vertex with the minimum degree density and cloning the
vertex with the maximum degree density would result in an $n$ vertex
graph that is $\F$-free and has more edges than $M_n$. It is trivial
to check that the average degree density is $d(M_n)$ and hence every
vertex in $M_n$ has a degree density between $d(M_n)+2/(n-1)$ and
$d(M_n)-2/(n-1)$. We can take $M_n$ and randomly remove each of its
edges with probability $1-\rho/d(M_n)$ (recall that
$\rho\leq\pi(\F)\leq d(M_n)$ so this is a valid probability). An
application of Chebyshev's inequality shows that there must exist a
way of removing edges from $M_n$ such that $(1-o(1))n$ of the
vertices will have a degree density that is at most $
n^{-\frac{1}{2}}$ from their expected degree density (which is their
degree density in $M_n$ multiplied by $\rho/d(M_n)$). In other words
$(1-o(1))n$ of the vertices will have degree density $\rho+o(1)$ as
required.

Let $\sigma$ be the type of order $1$ (a single labelled vertex),
and Let $F$ be the $\sigma$-flag of order $3$ whose underlying graph
is a single edge (just to be clear here we are considering flags and
types to be simple uncoloured graphs). Since $G_n$ is almost regular
we know that for most choices of $\theta$ we have $p(F,\theta;G_n) =
\rho+o(1)$. In fact for any $\sigma$-flag $C$, we have
$p(C,F,\theta;G_n) = \rho p(C,\theta;G_n)+o(1)$ for $(1-o(1))n$
choices of $\theta$. Hence
\[
\mathbf{E}_{\theta\in\Theta}[p(C,F,\theta;G_n)-\rho p(C,\theta;G_n)]
= o(1)
\]
and the left hand side can be expressed explicitly in terms of
$p(H;G_n)$. Clearly for every $\sigma$-flag $C$ we can construct
such a constraint, which we will refer to as a \emph{regularity
constraint}.

By summing a linear combination of regularity constraints as well as
terms of the form $\mathbf{E}_{\theta\in\Theta}[\mathbf{p}_\theta^T
Q \mathbf{p}_\theta]$ we can create an expression which can be
written explicitly in terms of $p(H;G_n)$, and is asymptotically
non-negative provided $\rho \leq \pi(\F)$. For a given $\rho$ the
problem of finding the optimal linear combination of the regularity
constraints and positive semidefinite matrices $Q$ that minimize the
coefficients of $p(H;G_n)$ can be posed as a semidefinite
programming problem. If all the coefficients of $p(H;G_n)$ can be
made strictly negative a contradiction occurs implying $\rho$ must
in fact be an upper bound for $\pi(\F)$.

Note that finding the optimal upper bound of $\pi(\F)$ is not a
semidefinite programming problem but testing whether a specific
value is an upper bound is a semidefinite programming problem. It
can be proven (though we will not do so here) that if the
semidefinite program shows $\rho$ to be an upper bound then it will
also show $\rho'$ to be an upper bound for any $\rho'\geq\rho$. This
allows us to carry out a binary search to determine the optimal
bound to whatever accuracy we wish.

Extending this method to use partially defined graphs is a trivial
matter. The coefficients of $p(H;G_n)$ for the regularity
constraints can still be explicitly calculated even when $H$ is a
vertex-coloured $3$-graph. In fact calculating the probability of
the intersection of $C$ and $F$ appearing is equivalent to
considering the occurrence of the vertex-coloured flag formed from
$C$ by colouring the labelled vertex blue, and the other vertices
grey. This allows us to construct regularity constraints from larger
$\sigma$-flags, which in turn means we can add many more regularity
constraints into our semidefinite program. Because we are using
partially defined graphs we also need to include additional
constraints similar to those described in Section
\ref{SUBSEC:PartialExtra}. Since such constraints are of the form
$\sum_{H\in\mc{H}} a_H p(H;G)=0$ we can incorporate them into the
semidefinite program in the same way as the regularity constraints.

\section{Open problems and future work}

The power of using partially defined graphs comes from being able to
reduce the size of $\mc{H}$ to make computations more feasible.
Another way we could do this is the following. Let $\mc{H}$ be a
family of $\F$-free graphs we wish to reduce and let $\mc{H}'$ be a
family of subsets of $\mc{H}$ that partition $\mc{H}$, i.e.\
$\bigcup_{H'\in\mc{H}'} H' = \mc{H}$ and for all distinct
$H_1',H_2'\in\mc{H}'$ we have $H_1'\cap H_2'=\emptyset$. Now instead
of applying Razborov's semidefinite flag algebra technique to
$H\in\mc{H}$ we can apply it to $H'\in\mc{H'}$. We can define
$p(H';G)$ to be $\sum_{H\in H'}p(H;G)$. There is some difficulty in
expressing quantities like the products of flags
$\mathbf{E}_{\theta\in\Theta}[p(F_a,F_b,\theta;G)]$ in terms of a
linear sum of $p(H';G)$. However, we can overcome this by bounding
the coefficients above and below, for example
\begin{align}\label{Eq:UpperLower}
\sum_{H'\in\mc{H}'} c_{min, H'} p(H';G)\leq
\mathbf{E}_{\theta\in\Theta}[p(F_a,F_b,\theta;G)] \leq
\sum_{H'\in\mc{H}'} c_{max, H'} p(H';G)
\end{align}
where
\[
c_{min,H'}=\min_{H\in
H'}\mathbf{E}_{\theta\in\Theta_H}[p(F_a,F_b,\theta;H)]
\]
and
\[
c_{max,H'}=\max_{H\in
H'}\mathbf{E}_{\theta\in\Theta_H}[p(F_a,F_b,\theta;H)].
\]
In order to keep quantities like
$\mathbf{E}_{\theta\in\Theta}[\mathbf{p}_\theta^TQ\mathbf{p}_{\theta}]$
non-negative we have to be careful about whether we apply the upper
or lower bound given in (\ref{Eq:UpperLower}). We can do this and
still represent the problem as a semidefinite program, by replacing
the matrix $Q$ with two matrices $Q_+$ and $Q_-$ such that $Q =
Q_+-Q_-$ and the entries of $Q_+$ and $Q_-$ are all non-negative.
Consequently
$\mathbf{E}_{\theta\in\Theta}[\mathbf{p}_\theta^TQ\mathbf{p}_{\theta}]
=
\mathbf{E}_{\theta\in\Theta}[\mathbf{p}_\theta^TQ_+\mathbf{p}_{\theta}]
-\mathbf{E}_{\theta\in\Theta}[\mathbf{p}_\theta^TQ_-\mathbf{p}_{\theta}]$
and we can ensure non-negativity by applying upper bounds to terms
coming from
$\mathbf{E}_{\theta\in\Theta}[\mathbf{p}_\theta^TQ_+\mathbf{p}_{\theta}]$
and lower bounds to terms coming from
$\mathbf{E}_{\theta\in\Theta}[\mathbf{p}_\theta^TQ_-\mathbf{p}_{\theta}]$.

It is unclear whether this method will ever produce any improvements
over the partially defined graph version of Razborov's method. Also
it is not apparent what a good choice for the partition $\mc{H}'$
is. Nevertheless it may be worth further investigation.

The study of Tur\'an problems in hypercubes is largely motivated by
Erd\H os' conjecture that $\pi_e(\mc{Q}_2)=1/2$. This is perhaps the
most interesting question in the area, and still remains open. We
have provided improvements on the bounds of various edge and vertex
Tur\'an densities but were only able to calculate $\pi_v(R_2)$
exactly. Improving the bounds further to get exact results in any of
the problems discussed would be of interest. Also any significant
improvement of the bound of $\pi(K_4^3)$ would be interesting.

\section{Acknowledgements}
The results in Sections \ref{SUBSEC:VertexHypercube} and
\ref{SUBSEC:EdgeHypercube} originally appeared in the author's PhD
thesis in March 2011, see \cite{B11}. J\'ozsef Balogh, Ping Hu,
Bernard Lidick\'y, and Hong Liu independently rediscovered and
verified the weaker edge Tur\'an density results that do not use
partial hypercubes (Theorem \ref{THM:SquareFree}) in December 2011,
see \cite{BHLL}.

%Their paper was later updated in May 2012 (see \cite{BHLL}) in which
%they acknowledge the proofs given in \cite{B11}, and also apply
%Razborov's method to forbidding posets in the boolean lattice.

The author would like to thank John Talbot for his helpful comments
about the paper, and for running the calculations for the new
$\pi(K_4^3)$ bound on his computer (the author's own seven year old
computer was not up to the task).

\end{document}